\documentclass{base_class}
\usepackage{bbm}

\title[]{Godbersen's conjecture and the $L_p$-Rogers--Shephard inequality} 

\author{Jan Kotrbat{\'y}}
\author{Mohamed A. Mouamine}
\thanks{Supported by Charles University grants PRIMUS/24/SCI/009 and UNCE/24/SCI/022.}
 
\address{Charles University, Faculty of Mathematics and Physics, Mathematical Institute of Charles University, Sokolovsk\'a 49/83, 186 00 Prague, Czechia}
\email{kotrbaty@karlin.mff.cuni.cz}
\email{mohamed-abdeldjalil.mouamine@matfyz.cuni.cz}
\subjclass[2020]{52A40, 52A39, 52B11}
\date{\today}

\begin{document}

\maketitle

\begin{abstract} 
We prove that the mixed volume of a convex body of fixed positive volume with its reflection about the origin is maximized by simplices. This confirms a conjecture of C.~Godbersen from 1938 and refines the classical Rogers--Shephard inequality. We also prove that simplices are the only extremizers among convex polytopes. Finally, we use this inequality to prove an $L_p$-version of the Rogers--Shephard inequality for convex bodies containing the origin and show that, for any $p\in(1,\infty]$, the only extremizers are simplices with a vertex at the origin.
\end{abstract}

\section{Introduction}

\subsection{Background}

The celebrated Rogers--Shephard inequality \cite{RogersShephard57,RogersShephard58} gives a sharp upper bound on the volume of the difference body $K-K=K+(-K)$ associated to a convex body $K\subset\RR^n$:
\begin{align}
\label{eq:RS}
\vol_n(K-K)\leq {2n\choose n}\vol_n(K).
\end{align}
Here $\vol_n$ is the Lebesgue measure on $\RR^n$, $K+L=\{x+y\mid x\in K, y\in L\}$ is the Minkowski sum of $K,L\subset\RR^n$, and $-K=\{-x:x\in K\}$ is the reflection of $K\subset\RR^n$ about the origin. If $K$ has non-empty interior, equality in \eqref{eq:RS} holds if and only if $K$ is a simplex. 

An alternative proof of the Rogers--Shephard inequality was given by Chakerian \cite{Chakerian}. Based on Chakerian's approach, Schneider \cite{Schneider70} proved the following higher-order generalization of the Rogers--Shephard inequality:
\begin{align}
\label{eq:Schneider}
\vol_{rn}(\Delta_r K-K^r)\leq {rn+n\choose n}\vol_n(K)^r,
\end{align}
where $r\in\NN$ and $\Delta_r:\RR^n\to\RR^{rn}$ is the diagonal embedding. Assuming $\interior K\neq \emptyset$, equality in \eqref{eq:Schneider} holds  again precisely for simplices. Schneider's work laid foundations for much of the recent progress in higher-order Brunn--Minkowski theory \cite{HaddadETAL23, LangharstETAL:mOrder,  LangharstSolaUlivelli24, HaddadETAL25, LangharstRoysdonZhao25,LangharstXi24, HaddadETAL25+,Langharst:Comments,ZhouETAL25, K25, FK26,KW22}. For further generalizations of the Rogers--Shephard inequality, in particular to the functional setting, see  \cite{Roysdon20,AG19,AG21,AG16,Colesanti06}.

The Rogers--Shephard inequality can be generalized in yet another way, as we now discuss. According to the classical result of Minkowski, there exists a symmetric, Minkowski multilinear function $V(K_1,\dots,K_n)$ on the space of convex bodies, called  mixed volume, with the property that $V(K,\dots,K)=\vol_n(K)$. Expanding the left-hand side of \eqref{eq:RS} by multilinearity and the right-hand side using ${2n\choose n}=\sum_{k=0}^n{n\choose k}^2$,
 it is natural to expect that the inequality holds term-wise; namely, that for all $0\leq k\leq n$ one has
 \begin{align}
 \label{eq:Godbersen0}
 V(K[k],-K[n-k])\leq  {n\choose k}\vol_n(K),
 \end{align}
 where $K[k]$ denotes $k$ copies of convex body $K\subset \RR^n$. This conjecture was first formulated by Godbersen \cite{Godbersen} and later independently by Makai Jr. \cite{Makai}, see also \cite[\S23.1.3]{BuragoZalgaller} and \cite[\S10.1]{Schneider2014}. Since it is well known that equality holds in \eqref{eq:Godbersen0} if $K$ is a simplex (see, e.g., \cite[\S4]{RogersShephard57}), it is also commonly conjectured that simplices are the only non-trivial extremizers.

Apart from the trivial cases $k\in\{0,n\}$, the conjecture is known for general convex bodies only if $k\in\{1,n-1\}$. In this case, \eqref{eq:Godbersen0} follows at once from the monotonicity of the mixed volume and the well-known fact that $-K\subset nK$ holds for any convex body $K\subset\RR^n$ with centroid at the origin, while the characterization of equality cases follows from the fact that a simplex is the unique extremizer of the so-called Minkowski measure of symmetry, see \cite[\S6.1]{Grunbaum63}. In particular, the conjecture is true in dimensions $n\leq 3$. For $n\geq4$, it has been established only for certain special classes of convex bodies. First, Godbersen \cite{Godbersen} proved \eqref{eq:Godbersen0} for bodies of constant width. Much more recently, Artstein-Avidan, Sadovsky, and Sanyal \cite{ArtsteinSadovskySanyal23} established \eqref{eq:Godbersen0} for anti-blocking convex bodies and proved that simplices are the only extremizers within this class. Sadovsky \cite{Sadovsky25} subsequently extended this result to the class of locally anti-blocking bodies.

Besides these partial results, several related results have been obtained, providing additional evidence for the validity of \eqref{eq:Godbersen0}. First, Artstein-Avidan, Einhorn, Florentin, and Ostrover \cite{ArtsteinETAL15} proved that for $1\leq k\leq n$,
\begin{align}
\label{eq:bound}
V(K[k],-K[n-k])\leq \frac{n^n}{k^k(n-k)^{n-k}}\vol_n(K)\sim{n\choose k}\sqrt{2\pi\frac{k(n-k)}n}\vol_n(K),
\end{align}
providing the currently best known upper bound on $V(K[k],-K[n-k])$. They also proposed a stronger conjecture that interpolates between \eqref{eq:Godbersen0} and another long-standing conjecture of Fáry and Rédei \cite{FaryRedei}. More recently, Artstein-Avidan and Putterman \cite{ArtsteinPutterman25} proved a refinement of \eqref{eq:bound} from which they concluded that, in particular, \eqref{eq:Godbersen0} holds ``on average'' in the sense that
\begin{align*}
\frac1{n+1}\sum_{k=0}^n\frac{V(K[k],-K[n-k])}{{n \choose k}}\leq \vol_n(K).
\end{align*}
Moreover, they conjectured that for any $\lambda\in(0,1)$, the ratio
\begin{align*}
\frac{\vol_n\big( (1-\lambda)K+\lambda(-K)\big)}{\vol_n(K)}
\end{align*}
is maximized precisely by $n$-simplices, and verified this for $n\leq5$. Observe that their conjecture would follow from the inequality \eqref{eq:Godbersen0} and the known characterization of equality cases for $k=1$; on the other hand, it subsumes the Rogers--Shephard inequality \eqref{eq:RS} for $\lambda=\frac12$.

Finally, a higher-order generalization of \eqref{eq:Godbersen0} refining Schneider's inequality \eqref{eq:Schneider} was proposed by Schneider \cite{Schneider00} and independently by the first-named author \cite{K25}. The first-named author \cite{K25} recently proved this more general conjecture for the class of anti-blocking convex bodies. Further special cases were very recently verified by Fryš and the first-named author \cite{FK26} who also generalized to the higher-order setting the results and conjectures of Artstein-Avidan--Einhorn--Florentin--Ostrover \cite{ArtsteinETAL15} and Artstein-Avidan--Putterman \cite{ArtsteinPutterman25}.

\subsection{Our results}

In this paper we settle the conjectured inequality of Godbersen--Makai Jr. in the affirmative and prove that simplices are the only extremizers among convex polytopes. Thus, our first main result is the following

\begin{theorem}
\label{thm:mainA}
For each convex body $K\subset\RR^n$ and each $0\leq k\leq n$ one has
\begin{align}
\label{eq:Godbersen}
V(K[k],-K[n-k])\leq  {n\choose k}\vol_n(K).
\end{align}
Moreover, if $0<k<n$ and $K$ is a polytope with $\interior K\neq\emptyset$, equality holds in \eqref{eq:Godbersen} if and only if $K$ is a simplex.
\end{theorem}

Our proof of Theorem \ref{thm:mainA} is directly motivated by the recent developments related to the higher-order generalization of Godbersen's conjecture \cite{K25,FK26}. Namely, we observe that \eqref{eq:Godbersen} is in fact a consequence of the monotonicity of suitable mixed volume in dimension $2n$. This is done in Section \ref{s:inequality}, see Proposition \ref{pro:Godbersen}. Moreover, using the same argument, we give a new simple proof of the Rogers--Shephard inequality \eqref{eq:RS} in Proposition \ref{pro:RS}.

While the classification of equality cases of the monotonicity of mixed volumes is in general a difficult open problem (see \cite{Schneider85,vanHandelWang} and references therein), it is well understood in the case of convex polytopes. Using this characterization, we show  in Section \ref{s:equality} that if a polytope $K$ satisfies equality in \eqref{eq:Godbersen}, its faces necessarily obey enough incidence relations to eventually force $K$ to be a simplex.

Let us point out that our proof of \eqref{eq:Godbersen} directly generalizes to the higher-order case, which settles the inequality part of \cite[Conjecture 1.3]{K25} and provides a new simple proof of Schneider's inequality \eqref{eq:Schneider}. Furthermore, we expect that our method will establish the characterization of equality cases for polytopes also in this more general situation. This will be worked out in subsequent work.

As discussed above, Theorem \ref{thm:mainA} directly implies \cite[Theorem 1.4]{ArtsteinETAL15} and \cite[Theorem 4]{ArtsteinPutterman25}, and confirms \cite[Conjecture 2]{ArtsteinPutterman25}. Moreover, Theorem \ref{thm:mainA} substantially simplifies the proof of the quantitative stability version of the Rogers--Shephard inequality due to  B\"or\"oczky \cite{boroczky05}, since it reduces the problem to that of the stability of \eqref{eq:Godbersen} for $k=1$ and hence to that of the stability of the Minkowski measure of symmetry, see also \cite{boroczky96} and \cite{Schneider09}.

We further use Theorem \ref{thm:mainA} to resolve another long-standing open problem, see \cite[p. 532]{Schneider2014}. Namely, we show that \eqref{eq:Godbersen} implies a version of the Rogers--Shephard inequality in the Firey--Lutwak $L_p$-Brunn--Minkowski theory \cite{firey62,LutwakI,LutwakII}. This arises when the usual Minkowski sum in \eqref{eq:RS} is replaced by the $L_p$-Minkowski addition for $p>1$, see \S\ref{ss:Lp} for definition and more details.  Recall also that for $0\leq y\leq x$ we define ${x\choose y}:=\frac{\Gamma(x+1)}{\Gamma(y+1)\Gamma(x-y+1)}$. Then our second main result is

\begin{theorem}
\label{thm:mainB}
Let $p\in(1,\infty]$ and $q\in[1,\infty)$ satisfy $\frac1p + \frac1q =1$. For each convex body $K\subset\RR^n$ with $0\in K$ one has
\begin{align}
\label{eq:LpRS}
\vol_n\big(K+_p(-K)\big)\leq \sum_{k=0}^n\binom{n}{k}^2\binom{n/q}{k/q}^{-1}\vol_n(K).
\end{align}
If $\interior K\neq\emptyset$, equality holds in \eqref{eq:LpRS} if and only if $K$ is a simplex with one vertex at the origin.
\end{theorem}

Theorem \ref{thm:mainB}, which we prove in Section 5, was previously known only in the following special cases. First, for $n=2$, \eqref{eq:LpRS} was proven by Bianchini and Colesanti \cite{BianchiniColesanti}. Second, Manui, Ndiaye, and Zvavitch \cite{ManuiNdiayeZvavitch24} verified \eqref{eq:LpRS} for locally anti-blocking convex bodies. Third, very recently, Fradelizi, Manui, Meyer, and Ndiaye \cite{Fradelizi:Equality} proved that in both these previously established cases, equality in \eqref{eq:LpRS} holds precisely for simplices with a vertex at the origin.

Finally, as pointed out to us by A.~Manui, our proof of Theorem \ref{thm:mainB} also  immediately yields a version of the $L_p$-Rogers--Shephard inequality for symmetric convex bodies, see Proposition \ref{pro:LpRSsym}. This confirms \cite[Conjecture 4]{Fradelizi:LpRS} of Fradelizi, Manui, Meyer, and Ndiaye who previously verified it for  zonoids.

\subsection*{Acknowledgments} We are grateful to M.~Szusterman for numerous helpful comments and suggestions that significantly improved the presentation of the paper. We also thank A.~Manui for highlighting the connection with the $L_p$-Rogers--Shephard inequality for symmetric bodies, E.~Putterman for stimulating discussions and for pointing out reference \cite{develin}, and K.~B\"or\"oczky, F.~Fryš, F.~Mussnig, and T.~Wannerer for their encouragement and useful comments. 

\section{Preliminaries}

Let us start by collecting several general facts from convex geometry that we use in the paper. Our references are the standard monographs \cite{Schneider2014,Grunbaum2003,Ziegler1995}. Throughout the paper, we assume that $n\geq 2$. 

\subsection{Convex bodies and convex polytopes}

Convex body is a non-empty compact convex set $K\subset\RR^n$. We will denote the set of all convex bodies by $\calK(\RR^n)$. The dimension of a convex body $K\in\calK(\RR^n)$ is, by definition, the dimension of its affine hull; we write $\dim K:=\dim\aff K$. Clearly $\dim K=n$ if and only if $\interior K\neq\emptyset$.

The support function $h_K(u):=\sup_{x\in K}\ip xu$, $u\in\RR^n$, of a convex body $K\in\calK(\RR^n)$ is sublinear, that is, for any $u,v\in\RR^n$ and $\lambda\geq0$ one has
$$
h_K(\lambda u)=\lambda h_K(u)\quad\text{and}\quad h_K(u+v)\leq h_K(u)+ h_K(v),
$$
and satisfies $h_K \leq h_L$ if $K \subseteq L$. The gauge function $g_K(x) := \inf\{\lambda \geq 0 \mid x \in \lambda K\}$, $x \in \RR^n$, of a convex body $K\in\calK(\RR^n)$ containing the origin is convex. For any convex body $K\in\calK(\RR^n)$ and $u\in\RR^n$, the set
$$
F_K(u):= \{x \in K\mid \ip{x}{u} = h_K(u)\}
$$
is called the face of $K$ with normal direction $u$. For a face $F\subset K$, we define its normal cone
$$
N_K(F) := \{u \in \RR^n \mid F \subseteq F_K(u)\}.
$$

A (convex) polytope is a convex body that arises as the convex hull of finitely many points in $\RR^n$ or, equivalently, as a bounded intersection of finitely many closed half-spaces. Observe that a face of a polytope is itself a polytope.

A simplex is the convex hull of affine independent points in $\RR^n$. We will need the following characterization of simplices due to Develin:
\begin{lemma}[Develin {\cite[Theorem 1]{develin}}]
\label{lem:develin}
Let $1\leq k\leq n-1$ and let $K\in\calK(\RR^n)$ be a polytope with $\interior K\neq\emptyset$. Then $K$ is a simplex if and only if each $k$-dimensional face of $K$ and each $(n-k)$-dimensional face of $K$ intersect.
\end{lemma}

\subsection{Mixed volume and mixed surface area measure}

Let $\vol_n$ be the canonical Lebesgue measure on $\RR^n$.
For convex bodies $K_1,\dots,K_m \in \calK(\RR^n)$ and $\lambda_1,\dots,\lambda_m\geq0$, we have 
$$
\vol_n(\lambda_1 K_1 + \cdots + \lambda_m K_m) = \sum_{i_1,\dots,i_n = 1}^m \lambda_{i_1}\cdots\lambda_{i_n} V(K_{i_1},\dots,K_{i_n}).
$$ 
The coefficients $V(K_{i_1},\dots,K_{i_n})$ of this polynomial are called mixed volumes. They satisfies for any $L,K_1,\dots,K_n\in\calK(\RR^n)$, permutation $\rho$ of $\{1,\dots,n\}$,  and $\lambda\geq0$ the following properties:
\begin{align*}
V(L,\dots,L)&=\vol_n(L),\\
V\big(K_{\rho(1)},\dots,K_{\rho(n)}\big)&=V(K_1,\dots,K_n),\\
V(K_1+\lambda L,K_2,\dots,K_n)&=V(K_1,\dots,K_n)+\lambda V(L,K_2,\dots,K_n).
\end{align*}

The following statement which is essentially an application of Fubini's theorem will be crucial for our proof of \eqref{eq:Godbersen}.

\begin{lemma}[{\cite[Theorem 5.3.1]{Schneider2014}}]
\label{lem:exact}
Assume $1\leq k\leq N-1$ and let $E\subset\RR^N$ be a $k$-dimensional subspace. For any $L_1,\dots,L_k,M_1,\dots,M_{N-k}\in\calK(\RR^N)$ with $L_i\subset E$, $i=1,\dots,k$, one has
\begin{align*}
{N\choose k}V(L_1,\dots,L_k,M_1,\dots,M_{N-k})=V_E(L_1,\dots,L_k)V_{E^\perp}(M_1|E^\perp,\dots,M_{N-k}|E^\perp).
\end{align*}
Here $M|E^\perp$ denotes the orthogonal projection of $M\subset\RR^N$ onto $E^\perp$, and the mixed volumes $V_E$ and $V_{E^\perp}$ are normalized by the canonical Lebesgue measures on the Euclidean spaces $E\subset\RR^N$ and $E^\perp\subset\RR^N$, respectively.
\end{lemma}

For any $K_1,\dots,K_{n-1}\in\calK(\RR^n)$ there further exists a unique non-negative and finite Borel measure $S(K_1,\dots,K_{n-1};\bullet)$ on $S^{n-1}$, called the mixed surface area measure, such that for any $L\in\calK(\RR^n)$ it holds
\begin{align}
\label{eq:MVint}
V(L,K_1,\dots,K_{n-1}) = \frac{1}{n} \int_{S^{n-1}} h_L(u) \d S(K_1,\dots,K_{n-1}; u).
\end{align}
This integral representations immediately yields the monotonicity of the mixed volume: If $K\subset L$, one has $V(K,K_1,\dots,K_{n-1})\leq V(L,K_1,\dots,K_{n-1})$. In particular, $V\geq0$.

\begin{lemma}[{\cite[Theorem 5.1.8]{Schneider2014}}]
\label{lem:Vpositive}
For $K_1,\dots,K_n\in\calK(\RR^n)$, one has $V(K_1,\dots,K_n)>0$  if and only if $\dim(K_I)\geq |I|$ for any $I\subset\{1,\dots,n\}$, where $K_I:=\sum_{i\in I}K_i$.
\end{lemma}

In particular, for any $K,K_1,\dots,K_{n-k}\in\calK(\RR^n)$, one has $V(K[k],K_1,\dots,K_{n-k})=0$ whenever $k>\dim K$. Here $K[k]$ denotes $k$ copies of $K$.

If $K_1,\dots,K_{n-1}\in\calK(\RR^n)$ are polytopes, then the mixed surface are measure is atomic; more precisely, 
\begin{align*}
S(K_1,\dots,K_{n-1};\bullet)=\sum_u V_{n-1}\big(F_{K_1}(u),\dots,F_{K_{n-1}}(u)\big)\delta_u,
\end{align*}
where the sum is over all outer unit normals to the facets of $K_1+ \cdots + K_{n-1}$, the mixed volume $V_{n-1}$ is taken on $u^\perp+u\cong\RR^{n-1}$, and $\delta_u$ is the Dirac measure supported in $u$, see \cite[(5.22)]{Schneider2014}. This together with Lemma \ref{lem:Vpositive} yields at once an explicit description of the support of $S$:
\begin{align}
\label{eq:supp}
	\supp S(K_1,\dots,K_{n-1}; \bullet) = \{ u \in S^{n-1} \mid \dim \sum_{i\in I} F_{K_i}(u) \geq |I| \text{ for every } I \subseteq [n-1]\},
\end{align} 
where we also used that $F_{K_i+K_j}(u)=F_{K_i}(u)+F_{K_j}(u)$, see \cite[Theorem 1.7.5]{Schneider2014}. Formula \eqref{eq:supp} will be used extensively in our proof of the characterization of equality cases of \eqref{eq:Godbersen} for polytopes.

\subsection{$L_p$-Minkowski sum}
\label{ss:Lp}

Fix $p \in[1,\infty]$ and consider convex bodies $K,L\in\calK(\RR^n)$ containing the origin. The $L_p$-Minkowski sum of $K$ and $L$ is the convex body $K+_p L$ defined by 
\begin{align*}
h_{K+_p L}(x) = \left(h_K^p(x) + h_L^p(x)\right)^{\frac1p}.
\end{align*}
Note that $K+_1 L=K+L$ and $K+_\infty L=\conv (K\cup L)$. If $p>1$, one has
\begin{align}
\label{eq:+Lp}
K +_p L = \left\{(1-t)^{\frac1q} x + t^{\frac1q} y\mid x \in K, y \in L, t \in [0,1]\right\},
\end{align}
where $q\in[1,\infty)$ is the H\"older conjugate of $p$, i.e.,
$$
\frac1p + \frac1q =1,
$$
see \cite[Lemma 2]{LYZ12}.

\section{Proof of Godbersen's inequality}
\label{s:inequality}

In this section we will prove the first part of Theorem \ref{thm:mainA}; namely, we will establish the inequality \eqref{eq:Godbersen} for general convex bodies.

Consider the diagonal embedding $\Delta:\RR^n\to \RR^{2n}$ given by $\Delta(x)=(x,x)$. In what follows we will use Lemma \ref{lem:exact} for $E:=\Delta(\RR^n)\subset\RR^{2n}$. To this end, observe that for $K_1,\dots,K_n\in\calK(\RR^n)$ one has
\begin{align}
\label{eq:volDelta}
V_E(\Delta K_1,\dots,\Delta K_n)= \sqrt2^nV(K_1,\dots,K_n).
\end{align}
The constant $\sqrt2^n$ can be determined, e.g., by taking $K_1,\dots,K_n$ to be pairwise orthogonal unit segments. Similarly, for $\wt\Delta:\RR^n\to \RR^{2n}$, $\wt\Delta(x)=(x,-x)$, one has $E^\perp=\wt\Delta(\RR^n)$  and
\begin{align}
\label{eq:volDeltaT}
V_{E^\perp}(\wt\Delta K_1,\dots,\wt\Delta K_n)= \sqrt2^nV(K_1,\dots,K_n).
\end{align}
We will also consider the map $\delta:\RR^{2n}\to\RR^n $ given by $\delta(x,y)= x-y$. Observe that then
$$
\frac12 \wt\Delta\circ\delta:\RR^{2n}\to E^\perp
$$
is  the orthogonal projection. In particular, using \eqref{eq:volDeltaT}, for any $M_1,\dots,M_n\in\calK(\RR^{2n})$ we have
\begin{align}
\label{eq:volproj}
V_{E^\perp}(M_1|E^\perp,\dots,M_n|E^\perp)= \frac1{\sqrt2^n}V(\delta M_1,\dots,\delta M_n).
\end{align}

We first give a simple proof of the Rogers--Shephard inequality.

\begin{proposition}
\label{pro:RS}
For each $K\in\calK(\RR^n)$ one has
$$
\vol_n(K-K)\leq{2n\choose n}\vol_n(K).
$$
\end{proposition}

\begin{proof}
Clearly, it is enough to consider the case $\vol_n(K)>0$. Using the monotonicity of the mixed volume applied to $\Delta K\subset K^2$, Lemma \ref{lem:exact}, \eqref{eq:volDelta}, and \eqref{eq:volproj}, we have
\begin{align*}
{2n\choose n}\vol_n(K)^2&={2n\choose n}\vol_{2n}(K^2)\\
&\geq {2n\choose n}V(\Delta K[n],K^2[n])\\
&=\vol_n(K)\vol_n(\delta K^2)\\
&=\vol_n(K)\vol_n(K-K)
\end{align*}
and the claim follows.
\end{proof}

A slight modification of the same argument yields the inequality part of Theorem \ref{thm:mainA}. For the proof we will consider the maps $\iota_1,\iota_2:\RR^n\to\RR^{2n}$ given by $\iota_1(x)=(x,0)$ and $\iota_2(x)=(0,x)$, respectively.

\begin{proposition}
\label{pro:Godbersen}
For each $K\in\calK(\RR^n)$ each $0\leq k\leq n$ one has
$$
V(K[k],-K[n-k])\leq{n\choose k}\vol_n(K).
$$
\end{proposition}

\begin{proof}
We may assume  $\vol_n(K)>0$. On the one hand, using  $K^2=\iota_1K+\iota_2K$ together with Lemma \ref{lem:exact},  we have
$$
{2n\choose n}V(K^2[n],\iota_1 K[k],\iota_2K[n-k])={2n\choose n}{n\choose k}V(\iota_1 K[n],\iota_2K[n])={n\choose k}\vol_n(K)^2.
$$
On the other hand, using the monotonicity of the mixed volume, Lemma \ref{lem:exact}, \eqref{eq:volDelta}, and \eqref{eq:volproj}, we have
\begin{align*}
{2n\choose n}V(K^2[n],\iota_1 K[k],\iota_2K[n-k])&\geq {2n\choose n}V(\Delta K[n],\iota_1 K[k],\iota_2K[n-k])\\
&=\vol_n(K)V(\delta\iota_1K[k],\delta\iota_2 K[n-k])\\
&=\vol_n(K)V(K[k],-K[n-k]).
\end{align*}
Comparing the former and the latter, the claim follows.
\end{proof}

\section{Equality cases of Godbersen's inequality for convex polytopes}
\label{s:equality}

Now we will characterize the equality cases of Godbersen's inequality within the class of convex polytopes, thus proving the second part of Theorem \ref{thm:mainA}.

Recall that the only inequality used in the proof of Proposition \ref{pro:Godbersen} was the monotonicity of the mixed volume. Hence, $K\in\calK(\RR^n)$ with non-empty interior satisfies equality in \eqref{eq:Godbersen} if and only if
\begin{align}
\label{eq:equalities}
\begin{split}
V(\Delta K[n],\iota_1 K[k],\iota_2K[n-k])&=V(K^2[1],\Delta K[n-1],\iota_1 K[k],\iota_2K[n-k])\\
&=V(K^2[2],\Delta K[n-2],\iota_1 K[k],\iota_2K[n-k])\\
&\hspace{1.2ex}\vdots\\
&= V(K^2[n-1],\Delta K[1],\iota_1 K[k],\iota_2K[n-k])\\
&= V(K^2[n],\iota_1 K[k],\iota_2K[n-k]).
\end{split}
\end{align}

We will see that, if $1\leq k\leq n-1$ and $K$ is a polytope, the equalities \eqref{eq:equalities} yield certain incidence relations among the faces of $K$ that will eventually force the polytope to be a simplex. Note that the last equality in \eqref{eq:equalities} in fact always holds, and is therefore irrelevant for the characterization of equality cases of \eqref{eq:Godbersen}:

\begin{proposition}
\label{pro:n-1}
For any $K\in\calK(\RR^n)$ and $0\leq k \leq n$ one has
$$
V(K^2[n],\iota_1 K[k],\iota_2K[n-k])=V(K^2[n-1],\Delta K,\iota_1 K[k],\iota_2K[n-k]).
$$
\end{proposition}

\begin{proof}
Using $K^2=\iota_1K+\iota_2K$ and  Lemma \ref{lem:exact}, we can write
\begin{align*}
V(K^2[n-1],\Delta K,\iota_1 K[k],\iota_2K[n-k])&={n-1\choose n-k}V(\Delta K,\iota_1 K[n],\iota_2K[n-1])\\
&\qquad+{n-1\choose k}V(\Delta K,\iota_1 K[n-1],\iota_2K[n])\\
&={n\choose k}{2n\choose n}^{-1}\vol_n(K)^2\\
&={n\choose k}V(\iota_1K[n],\iota_2K[n])\\
&=V(K^2[n],\iota_1K[k],\iota_2[n-k]).
\end{align*}
\end{proof}

As we will see, the other equalities in \eqref{eq:equalities} will, by contrast, impose strong restrictions on the geometry of $K$. In order to make these restrictions explicit, we will rely heavily on the integral formula \eqref{eq:MVint}. To this end, let us first make two simple observations concerning the support functions and mixed surface area measures of the relevant convex bodies.

The following notation will be used for the rest of this section: For any $K\in\calK(\RR^n)$  and $u=(u_1,u_2)\in\RR^{2n}$ we will denote
$$
F_0:=F_K(u_1+u_2),\quad F_1:=F_K(u_1),\quad\text{and}\quad F_2:=F_K(u_2).
$$

\begin{proposition}
\label{pro:hK}
Let $K\in\calK(\RR^n)$. For any $u=(u_1,u_2)\in\RR^{2n}$, one has $h_{K^2}(u)=h_{\Delta K}(u)$ if and only if $F_1\cap F_2\neq\emptyset$.
\end{proposition}

\begin{proof}
This follows at once from the definition of a face, the sublinearity of support function $h_K$, and the fact that for any  $u=(u_1,u_2)\in\RR^{2n}$ one has
\begin{align}
\label{eq:hK2}
h_{K^2}(u)=h_K(u_1)+h_K(u_2)\qquad \text{and}\qquad h_{\Delta K}(u)=h_K(u_1+u_2).
\end{align}
\end{proof}

\begin{proposition}
\label{pro:supp}
Let $K\in\calK(\RR^n)$ be a polytope and let $a_0,a_1,a_2\geq0$ be integers satisfying $a_0+a_1+a_2=2n-1$. For any $u=(u_1,u_2)\in S^{2n-1}\subset\RR^{2n}$, one has
$$
u\in\supp S(\Delta K [a_0], \iota_1K[a_1], \iota_2K[a_2];\bullet)
$$ if and only if
\begin{align}
\label{eq:CharacterizationOfSupport}
\dim F_0\geq a_0,\quad \dim F_1\geq a_1,\quad \dim F_2\geq a_2,\quad\text{and}\quad \dim (\Delta F_0+\iota_1F_1+\iota_2F_2)=2n-1.
\end{align}
\end{proposition}

\begin{proof}
First, observe that
$$
F_{\Delta K}(u)=\Delta F_0,\quad F_{\iota_1 K}(u)=\iota_1 F_1,\quad\text{and}\quad F_{\iota_2 K}(u)=\iota_2 F_2.
$$
Second, if \eqref{eq:CharacterizationOfSupport} holds, then one also has $\dim(\Delta F_0+\iota_1 F_1)\geq a_0+a_1$ since  $\dim(\Delta F_0+\iota_1 F_1)=\dim F_0+\dim F_1$, and similarly $\dim(\Delta F_0+\iota_2 F_2)\geq a_0+a_2$ and $\dim(\iota_1 F_1+\iota_2 F_2)\geq a_1+a_2$. Now the claim follows from \eqref{eq:supp}.
\end{proof}

Combining these two observations with the integral representation of mixed volume \eqref{eq:MVint}, we obtain a first geometric criterion for equality in \eqref{eq:Godbersen} for polytopes.

\begin{lemma}\label{lem:supported_bad_direction_criterion}
Let  $1\leq k\leq n-1$ and let $K\in\calK(\RR^n)$ be a polytope. Then one has strict inequality in \eqref{eq:Godbersen} if and only if there exists $u=(u_1,u_2)\in\RR^{2n}$ such that 
\begin{align}
\label{eq:supported_bad_direction_criterion}
F_1 \cap F_2 = \emptyset, \quad  \dim F_1 \geq k, \quad \dim F_2 \geq n-k,\quad\text{and}\quad \dim (\Delta F_0+\iota_1F_1+\iota_2F_2)=2n-1.
\end{align}
\end{lemma}

\begin{proof}
Using Proposition \ref{pro:n-1}, the inequality in \eqref{eq:Godbersen} is strict if and only if
\begin{align}
\label{eq:VMstrict}
V(K^2[r+1], \Delta K[n-r-1],\iota_1K [k], \iota_2 K [n-k])>V(K^2[r], \Delta K[n-r],\iota_1K [k], \iota_2 K [n-k])
\end{align}
for some $0 \leq  r \leq n-2$. Using \eqref{eq:MVint} and $K^2 = \iota_1 K + \iota_2 K$, \eqref{eq:VMstrict} is equivalent to
$$
\sum_{i=0}^r \binom{r}{i}\eta_{r,k+i}>0,
$$
where we set
\[\eta_{r,j} :=\int_{S^{2n-1}} \big(h_{K^2}(u)-h_{\Delta K}(u)\big)\d S(\Delta K[n-r-1],\iota_1K[j],\iota_2K[n+r-j]; u).\]
Observe that $\eta_{r,j}\geq0$. Hence, the inequality in \eqref{eq:Godbersen} is strict if and only there exist $0\leq r\leq n-2$ and $k\leq j\leq k+r$ with $\eta_{r,j}>0$.

Assume that this is the case. Then there exists $u\in S^{2n-1}\subset\RR^{2n}$ satisfying
$$u
\in \supp S(\Delta K[n-r-1],\iota_1K[j],\iota_2K[n+r-j]; \bullet)
$$
and $h_{K^2}(u)>h_{\Delta K}(u)$. The former and Proposition \ref{pro:supp} imply
$$
\dim F_1\geq j\geq k,\quad \dim F_2\geq n+r-j\geq n-k,\quad\text{and}\quad \dim (\Delta F_0+\iota_1F_1+\iota_2F_2)=2n-1;
$$
the latter and Proposition \ref{pro:hK}  yield  $F_1\cap F_2=\emptyset$.

Conversely, assume there is $u\in\RR^{2n}$ satisfying \eqref{eq:supported_bad_direction_criterion}. Observe that $u\neq0$ since $u=0$ would imply $F_1=F_2=K$ and thus contradict $F_1\cap F_2=\emptyset$. By normalizing may therefore assume $u\in S^{2n-1}$. We want to show that $\eta_{r,j}>0$ for suitable $r$ and $j$. Since $F_1\cap F_2=\emptyset$ together with Proposition \ref{pro:hK} implies $h_{K^2}(u)>h_{\Delta K}(u)$, we only need to guarantee that
$$
u\in\supp  S(\Delta K[n-r-1],\iota_1K[j],\iota_2K[n+r-j]; \bullet).
$$
In view of Proposition \ref{pro:supp}, it thus suffices to find $0\leq r\leq n-2$ and $k\leq j\leq k+r$ with
\begin{align}
\label{eq:condition}
\dim F_0\geq n-r-1,\quad \dim F_1\geq j,\quad\text{and}\quad \dim F_2\geq n+r-j.
\end{align}
To this end, observe first that $ \dim F_0\geq 1$. Indeed, if $\dim F_0=0$, then $\dim(\iota_1F_1+\iota_2F_2)=2n-1$ and hence one of the faces $F_1$ and $F_2$ is $n$-dimensional which contradicts $F_1 \cap F_2 = \emptyset$. Now we will distinguish two cases. First, if $\dim F_0=n$, then \eqref{eq:condition} clearly holds for $r:=0$ and $j:=k$. Second, assume that $1\leq \dim F_0\leq n-1$. It is straightforward to verify that then
$$
r:=n-1-\dim F_0\quad\text{and}\quad j:=\max\{k,2n-1-\dim F_0-\dim F_2\}
$$
satisfy $\eqref{eq:condition}$ as well as $0\leq r\leq n-2$ and $k\leq j\leq k+r$. This finishes the proof.
\end{proof}

It turns out that the characterization of the equality cases in \eqref{eq:Godbersen} provided by Lemma \ref{lem:supported_bad_direction_criterion} can be further simplified. More precisely, the following lemma shows that the last condition in \eqref{eq:VMstrict} is in fact redundant.

\begin{lemma}\label{lem:maximal_partition_full_rank}
Let  $1\leq k\leq n-1$ and let $K\in\calK(\RR^n)$ be a polytope. Then one has strict inequality in \eqref{eq:Godbersen} if and only if there exists $u=(u_1,u_2)\in\RR^{2n}$ such that 
\begin{align}
\label{eq:supported_bad_direction_criterion_reduced}
F_1 \cap F_2 = \emptyset, \quad  \dim F_1 \geq k, \quad \text{and}\quad\dim F_2 \geq n-k.
\end{align}
\end{lemma}

\begin{proof}
Assume that there exists $u\in \RR^{2n}$ satisfying \eqref{eq:supported_bad_direction_criterion_reduced} and that it is chosen so that the number $\dim F_0+\dim F_1+\dim F_2$ is maximal possible. We will show that in such a case one has $\dim(\Delta  F_0 + \iota_1 F_1+ \iota_2 F_2)=2n-1$. The claim will then follow from Lemma \ref{lem:supported_bad_direction_criterion}.

We will need the following observation: Let $F$ be a face of $K$ and $x\in\RR^n$. If $x \in \relint N_K(F)$, then $F_K(x) = F $, while if $x \in N_K(F) \setminus \relint N_K(F)$, then $F \subsetneq F_K(x)$. Moreover, we always have $x\in \relint N_K(F_K(x))$. Observe also that $\linspan N_K(F)=\linspan(F-F)^\perp$. For $i=0,1,2$, we will denote 
$$
\b F_i:=\linspan(F_i-F_i)\quad\text{and}\quad  C_i:=N_K(F_i).
$$

As in the proof of Lemma \ref{lem:supported_bad_direction_criterion}, $F_1\cap F_2=\emptyset$ yields $u\neq 0$. Consider the cone
$$
\mathcal{C} := \{(v_1,v_2)\in\RR^{2n} : v_1 \in C_1, v_2 \in C_2, v_1+v_2 \in C_0 \}=N_{K^2}(\iota_1F_1+\iota_2F_2)\cap N_{\Delta K}(\Delta F_0).
$$
Then $u \in \relint\mathcal{C}$. Moreover, since $u\neq0$ lies in the relative interior of both cones $N_{\Delta K}(\Delta F_0)$ and $N_{K^2}(\iota_1F_1+\iota_2F_2)$, we have
\begin{align*}
\linspan \mathcal{C} &=\linspan N_{K^2}(\iota_1F_1+\iota_2F_2)\cap \linspan N_{\Delta K}(\Delta F_0)\\
&= \{(v_1,v_2)\in\RR^{2n}: v_1 \in \linspan C_1, v_2\in \linspan C_2, v_1+v_2\in \linspan C_0\},
\end{align*}

Let
$$
W:= \Delta \b F_0 + \iota_1\b F_1+ \iota_2\b F_2.
$$
We have $W\subset u^\perp$. Since $\dim(\Delta  F_0 + \iota_1 F_1+ \iota_2 F_2)=\dim W\leq \dim u^\perp=2n -1$, we need to show that $W=u^\perp$. Assume for contradiction that $W\subsetneq u^\perp$. Then there exists a non-zero vector $w=(w_1,w_2)\in W^\perp\cap u^\perp$. It necessarily satisfies
$$
w_1\in{\b F_1}^\perp=\linspan C_1,\quad w_2\in{\b F_2}^\perp=\linspan C_2,\quad\text{and}\quad w_1+w_2\in{\b F_0}^\perp=\linspan C_0.
$$
In particular, $w\in\linspan\calC$.

Because $C_i\cap(- C_i)=\{0\}$ for $i=0,1,2$, we have also $\calC\cap(-\calC)=\{0\}$ and therefore we can choose a linear functional $\xi\in(\RR^{2n})^*$ that is strictly positive on $\mathcal{C} \setminus\{0\}$. Define 
	$$H:=\{v\in\RR^{2n}\mid \xi(v) = \xi(u) \}\quad\text{and}\quad L := \mathcal{C} \cap H.$$
	We claim that $L$ is a (compact) polytope. Since $\calC$ is a closed polyhedral cone, it is only non-trivial to verify that $L$ is bounded. If $L$ was unbounded, there would exist a sequence $\{x_i\}\subset L$ with $x^i\to\infty$ for $i\to\infty$. By compactness and the choice of $\xi$, there exists
	$$
	m:=\min\{\xi(v):v\in \calC\cap S^{2n-1}\}>0.
	$$
	Observe also that $0\notin L$. Hence, we would get, for each $i\in\NN$,
	$$
	\xi(u)=\xi(x_i)=\Norm{x_i}\xi\left(\Norm{x_i}^{-1}x_i\right)\geq m\Norm{x_i}
	$$
	which contradicts $\Norm{x_i}\to\infty$.
	
	Denote
	\[\widehat w := w - \frac{\xi(w)}{\xi(u)}u.\]
	Then $\widehat w\in \linspan\calC$, $\xi(\widehat w) = 0$, and $\widehat w \neq 0$ since $w \in u^\perp$. Let
	$$
	\alpha:=\min\{t\in\RR\mid u + t \widehat w\in L\}\quad\text{and}\quad \beta:=\max\{t\in\RR\mid u + t \widehat w\in L\}.
	$$
Since $u\in\relint\calC$, we have $u\in\relint L$, and thus $\alpha<0<\beta$.

Consider the non-negative function
	$$
	D(v) = h_{K^2}(v)- h_{\Delta K}(v),\quad v\in\calC.
	$$
	Observe that $D$ is in fact the restriction of a linear functional. Indeed, pick $y_i \in F_0$, $i=0,1,2$. Then, using \eqref{eq:hK2} and the fact that $h_K(v_i)=\ip{y_i}{v_i}$ for all $v_i\in C_i$, $i=0,1,2$, we can write for any $v=(v_1,v_2)\in\calC$
	$$
	D(v)=\ip{y_1}{v_1}+\ip{y_2}{v_2}-\ip{y_0}{v_1+v_2}.
	$$
	If $D(u+\alpha\widehat w)=D(u+\beta\widehat w)=0$, then by linearity we would have $D(u)=0$ which would, according to Proposition \ref{pro:hK}, contradict the assumption $F_1 \cap F_2 = \emptyset$. Therefore, there exists a point $\tilde u=(\tilde u_1,\tilde u_2)\in\{u+\alpha\widehat w,u+\beta\widehat w\}$ on the relative boundary of $L$ such that $D(\tilde u)>0$. Since $\tilde u$ lies on the relative boundary of $\calC$, at least one of the vectors $\tilde u_1+\tilde u_2 \in C_0$, $\tilde u_1\in C_1$, and $ \quad \tilde u_2 \in C_2$ 	lies on the relative boundary of the corresponding normal cone. Consequently, at least one of the faces $F_K(\tilde u_1+\tilde u_2)$, $F_K(\tilde u_1)$, and $F_K(\tilde u_2)$ strictly contains $F_0$, $F_1$, or $F_2$ respectively. In particular, 
	$$
	\dim F_K(\tilde u_1+\tilde u_2)+\dim F_K(\tilde u_1)+\dim F_K(\tilde u_2)>\dim F_0+\dim F_1+\dim F_2.
	$$
	Observe also that $\dim F_K(\tilde u_1)\geq \dim F_1\geq k$ and  $\dim F_K(\tilde u_2)\geq \dim F_2\geq n-k$. Finally, since $D(\tilde u) > 0$, we have $F_K(\tilde u_1) \cap F_K(\tilde u_2) = \emptyset$ by Proposition \ref{pro:hK}, and we thus get a contradiction with the choice of $u$. We conclude that $W=u^\perp$ and hence $\dim(\Delta  F_0 + \iota_1 F_1+ \iota_2 F_2)=2n-1$, as desired.
\end{proof}

We have thus reduced the proof of the second part of Theorem \ref{thm:mainA} to the characterization of simplices due to Develin \cite{develin}, see Lemma \ref{lem:develin}.

\begin{proof}[Proof of Theorem \ref{thm:mainA}]
The inequality \eqref{eq:Godbersen} was proven in Proposition \ref{pro:Godbersen}. To prove the characterization of equality cases, assume that $1\leq k\leq n-1$ and $K\in\calK(\RR^n)$ is a polytope with $\interior K\neq \emptyset$. By Lemma \ref{lem:maximal_partition_full_rank}, the inequality  \eqref{eq:Godbersen} is strict if and only if there exist disjoint faces $F,G\subseteq K$ with $\dim F=k$ and $\dim G=n-k$. Indeed, observe that if \eqref{eq:supported_bad_direction_criterion_reduced} holds, then there exist faces $F\subseteq F_1\subseteq K$ and $G\subseteq F_2\subseteq K$, obviously still disjoint, with $\dim F=k$ and $\dim G=n-k$. Equivalently, equality holds in \eqref{eq:Godbersen} if and only if each $k$-dimensional face and each $(n-k)$-dimensional face of $K$ intersect. By Lemma \ref{lem:develin}, this is equivalent to $K$ being a simplex.

\end{proof}

\section{$L_p$-Rogers--Shephard inequality}
\label{s:Lp}

In the final section we will use Godbersen's inequality (Proposition \ref{pro:Godbersen}) to prove an $L_p$-version of the Rogers--Shephard inequality and  establish the characterization of its equality cases, thus proving Theorem \ref{thm:mainB}. Moreover, as a byproduct of our proof, we will obtain a sharpening of the inequality for symmetric convex bodies.

Let $p\in(1,\infty]$ and let $K\in\calK(\RR^n)$ be a convex body with $0\in K$. The following notation will be used throughout this section: First, we abbreviate
$$
D_pK:=K +_p (-K) .
$$
Second, for $t\in[0,1]$ we define
$$
C_t:=(1-t)^{\frac1q}K+t^{\frac1q}(-K),
$$
where $\frac1p + \frac1q = 1$. Then, according to \eqref{eq:+Lp},
$$
D_pK = \bigcup_{0 \leq t \leq 1} C_t.
$$
Finally, for $z\in D_pK$, we set
$$
I_z := \{t \in [0,1]: z \in C_t\}.
$$

\begin{lemma}
\label{lem:concave}
Let $p\in(1,\infty]$ and $K\in\calK(\RR^n)$ with $0\in K$. For any $z\in D_pK$, $I_z\subset[0,1]$ is a closed interval, and the function $z\mapsto |I_z|$ is concave on $D_pK$.
\end{lemma}

\begin{proof}
We first prove that the set
$$
\Sigma:=\{(z,t)\in\RR^n\times[0,1]\mid z\in C_t\}
$$
is convex. Let $(z_1,t_1),(z_2,t_2)\in\Sigma$, then there exist $x_1,x_2,y_1,y_2\in K$ with
$$
z_1=(1-t_1)^{\frac 1q}x_1+t_1^{\frac 1q}y_1\quad\text{and}\quad z_2=(1-t_2)^{\frac 1q}x_2+t_2^{\frac 1q}y_2.
$$
Take any $\lambda\in[0,1]$ and set
$$
t:=\lambda t_1+(1-\lambda)t_2,\quad z:=\lambda z_1+(1-\lambda)z_2, 
$$
and
$$
a_1:=\lambda (1-t_1)^{\frac 1q}(1-t)^{-\frac 1q},\quad a_2:=(1-\lambda) (1-t_2)^{\frac 1q}(1-t)^{-\frac 1q},\quad b_1:=\lambda t_1^{\frac 1q}t^{-\frac 1q},\quad b_2:=(1-\lambda) t_2^{\frac 1q}t^{-\frac 1q}.
$$
Since the function $t\mapsto t^{\frac 1q}$ is concave for $q\geq1$, we have
$$
a_1,a_2,a_1+a_2,b_1,b_2,b_1+b_2\in[0,1].
$$
Since $0\in K$, this implies $a_1x_1+a_2x_2,b_1y_1+b_2y_2\in K$. Consequently,
$$
z=(1-t)^{\frac 1q}(a_1x_1+a_2x_2)-t^{\frac1q}(b_1y_1+b_2y_2)\in C_t,
$$
proving that $(z,t)\in\Sigma$. Since $\Sigma$ is also clearly compact, we have $\Sigma\in\calK(\RR^{n+1})$.

Now $I_z$ is a (1-dimensional) section of $\Sigma$, hence a closed interval.

For $z\in D_p(K)$, denote $I_z=[a(z),b(z)]$, where $0\leq a(z)\leq b(z)\leq 1$. Let $z_1,z_2\in D_pK$ and $\lambda\in[0,1]$. Then by convexity of $\Sigma$ we have
$$
\big(\lambda z_1+(1-\lambda)z_2,\lambda a(z_1)+(1-\lambda)a(z_2)\big),
\big(\lambda z_1+(1-\lambda)z_2,\lambda b(z_1)+(1-\lambda)b(z_2)\big)\in\Sigma.
$$
Consequently, 
$$
[\lambda a(z_1)+(1-\lambda)a(z_2),\lambda b(z_1)+(1-\lambda)b(z_2)]\subset I_{\lambda z_1+(1-\lambda)z_2}
$$
and hence
$$
|I_{\lambda z_1+(1-\lambda)z_2}|\geq \lambda\big(b(z_1)-a(z_1)\big)+(1-\lambda)\big(b(z_2)-a(z_2)\big)= \lambda|I_{z_1}|+(1-\lambda)|I_{z_2}|,
$$
proving concavity.
\end{proof}

\begin{lemma}
\label{lem:volLp}
Let $p\in(1,\infty]$. For each $K\in\calK(\RR^n)$ with $0\in K$ one has
\begin{align}
\label{eq:volLp}
\vol_n(D_pK)\leq\left(1 + \frac nq\right) \int_0^1 \vol_n(C_t) \d t.
\end{align}
\end{lemma}

\begin{proof}
Fix $z\in D_pK$ and set $\tau:=g(z)$, where $g:=g_{D_pK}$ is the gauge function of $D_pK$. Assume that $z\neq 0$. Then $\tau\in(0,1]$ and $z\in\tau D_pK$. Hence, there exist $t_0 \in [0,1]$ and $x,y \in K$ such that 
$$
z = \tau (1-t_0)^{\frac1q}x - \tau t_0^{\frac1q} y.
$$
Choose any $t\in [\tau^qt_0,1 - \tau^q(1-t_0)]$. Then for 
$$
\alpha:=\tau(1-t_0)^{\frac1q}(1-t)^{-\frac1q}\quad \text{and} \quad\beta:= \tau t_0^{\frac1q}t^{-\frac1q}
$$
we have $\alpha,\beta\in[0,1]$. Consequently, using that $0\in K$, we have $\alpha x,\beta y \in K$ and hence
$$
z=(1-t)^{\frac1q}\alpha x - t^{\frac1q} \beta y\in C_t.
$$
Therefore, $[\tau^qt_0,1 - \tau^q(1-t_0)]\subset I_z$ which implies
\begin{align}
\label{eq:Iz}
|I_z| \geq 1 - g(z)^q.
\end{align}
Observe that since $I_0=[0,1]$ and $g(0)=0$, \eqref{eq:Iz} in fact holds for all $z\in D_pK$.

Integration using Fubini's theorem gives
\begin{align}
\label{eq:intgeq}
\int_0^1 \vol_n(C_t) \d t=\int_0^1\int_{\RR^n}  \mathbbm1_{C_t}(z) \d z\d t=\int_{D_pK}|I_z|\d z\geq \int_{D_pK}\big(1-g(z)^q\big)\d z.
\end{align}
Further, since for $t\in[0,1]$ one has $tD_pK=\{z \in \RR^n \mid g(z) \leq t\}\subset D_pK$, we compute
	\begin{align*}
	\int_{D_p K}g(z)^q \d z &= \int_{D_p K} \int_0^1 q t^{q-1}\mathbbm{1}_{\{g(z) > t\}} \d t \d z\\
	& = \int_0^1 q t^{q-1} \int_{D_pK} \mathbbm{1}_{\{g(z) > t\}} \d z \d t\\
	& = \int_0^1 q t^{q-1} \int_{D_pK} \left(1-\mathbbm{1}_{\{g(z) \leq t\}}\right) \d z \d t\\
	&= \vol_n(D_pK)\int_0^1 q t^{q-1}(1-t^n) \d t  \\
	&= \frac{n}{n+ q}\vol_n(D_p K).
	\end{align*}
	Plugging this into \eqref{eq:intgeq}, \eqref{eq:volLp} follows.
\end{proof}

\begin{proof}[Proof of Theorem \ref{thm:mainB}]
First, using Lemma \ref{lem:volLp},  Proposition \ref{pro:Godbersen}, and the formula
\begin{align}
\label{eq:integral}
\left(1 + \frac nq\right)\int_{0}^1(1-t)^{\frac kq}t^{\frac{n-k}q}\d t = \frac{\Gamma\left(1 + \frac kq\right)\Gamma\left(1 + \frac{n-k}{q}\right)}{\Gamma\left(1 + \frac{n}{q}\right)} = \binom{n/q}{k/q}^{-1},
\end{align}
we compute
\begin{align*}
\vol_n(D_pK)&\leq\left(1 + \frac nq\right) \int_0^1 \vol_n(C_t) \d t\\
&=\left(1 + \frac nq\right) \int_0^1 \left[\sum_{k=0}^n {n\choose k} (1-t)^{\frac k q}t^{\frac{n-k}q}V(K[k],-K[n-k]) \right]\d t\\
&\leq\left(1 + \frac nq\right)\vol_n(K) \sum_{k=0}^n {n\choose k}^2 \int_0^1 (1-t)^{\frac k q}t^{\frac{n-k}q} \d t\\
&=\sum_{k=0}^n\binom{n}{k}^2\binom{n/q}{k/q}^{-1}\vol_n(K),
\end{align*}
proving \eqref{eq:LpRS}.

Second, it is well known and not difficult to verify that equality in \eqref{eq:LpRS} holds if $K$ is a simplex with a vertex at the origin, see, e.g., \cite[Lemma 23]{ManuiNdiayeZvavitch24} and \cite[Theorem 1]{Fradelizi:Equality}.

Finally, assume that equality is attained  in \eqref{eq:LpRS} for $K\in\calK(\RR^n)$. Then, first, equality holds in \eqref{eq:Godbersen} for all $0\leq k\leq n$. In particular, for $k=1$, this implies that $K$ is a simplex. Second, by convexity of the gauge function and by Lemma \ref{lem:concave}, the function $z\mapsto|I_z|-1+g(z)^q$ is continuous on $D_pK$. By \eqref{eq:Iz}, for any $z\in D_pK$,
\begin{align}
\label{eq:equality}
|I_z|=1-g(z)^q .
\end{align}
Suppose  $0\in K$ is not an extreme point. Then there exists $u\in \RR^n\setminus\{0\}$ such that $u,-u\in K$. Take any $t\in[0,1]$ and set $s:=(1-t)^{\frac 1q}+t^{\frac 1q}$. Since $s\geq 1$ and $0\in K$, we have $s^{-1}u\in K\cap(-K)$ and hence
$$
u=(1-t)^{\frac1q}s^{-1}u+t^{\frac1q}s^{-1}u\in C_t.
$$
Since $t\in[0,1]$ was arbitrary, we have $|I_u|=1$. Since $u\neq 0$, we have $g(u)>0$, and thus we get
$$
|I_u|>1-g(u)^q,
$$
contradicting \eqref{eq:equality}. Therefore $0$ is an extreme point of the simplex $K$, hence a vertex.
\end{proof}

As pointed out to us by A. Manui, an immediate consequence of our proof of Theorem \ref{thm:mainB}, more precisely of Lemma \ref{lem:volLp}, is the following sharpening of the inequality \eqref{eq:LpRS} for symmetric convex bodies:

\begin{proposition}
\label{pro:LpRSsym}
Let $p\in(1,\infty]$ and $q\in[1,\infty)$ satisfy $\frac1p + \frac1q =1$. Let $K\in\calK(\RR^n)$ with $0\in K$. If $K$ has a center of symmetry, one has
\begin{align}
\label{eq:LpRSsym}
\vol_n(D_pK)\leq \sum_{k=0}^n\binom{n}{k}\binom{n/q}{k/q}^{-1}\vol_n(K).
\end{align}
\end{proposition}

\begin{proof}
Assume that $-K=K+x$ for some $x\in\RR^n$. Then for $t\in[0,1]$ we have
$$
C_t=\left((1-t)^{\frac 1q}+t^{\frac1q}\right)K+t^{\frac1q} x.
$$
Hence
$$
\vol_n(C_t)=\left((1-t)^{\frac 1q}+t^{\frac1q}\right)^n\vol_n(K)=\sum_{k=0}^n{n\choose k}(1-t)^{\frac kq}t^{\frac{n-k}q}\vol_n(K)
$$
and \eqref{eq:LpRSsym} follows at once from Lemma \ref{lem:volLp} and \eqref{eq:integral}.
\end{proof}

Proposition \ref{pro:LpRSsym} confirms \cite[Conjecture 4]{Fradelizi:LpRS}. Note that the authors of \cite{Fradelizi:LpRS} showed that the inequality is sharp, as equality in \eqref{eq:LpRSsym} is attained for the cube $K=[0,1]^n$ (and thus for all parallelotopes with a vertex at the origin), see \cite[Theorem 7]{Fradelizi:LpRS}.

\bibliographystyle{abbrv}
\bibliography{ref}

\begin{bibdiv}
\begin{biblist}

\bib{AG19}{article}{
      author={Alonso-Guti\'errez, David},
      author={Artstein-Avidan, Shiri},
      author={Gonz\'alez~Merino, Bernardo},
      author={Jim\'enez, Carlos~Hugo},
      author={Villa, Rafael},
       title={Rogers-{S}hephard and local {L}oomis-{W}hitney type
  inequalities},
        date={2019},
        ISSN={0025-5831,1432-1807},
     journal={Math. Ann.},
      volume={374},
      number={3-4},
       pages={1719\ndash 1771},
         url={https://doi.org/10.1007/s00208-019-01834-3},
      review={\MR{3985122}},
}

\bib{AG16}{article}{
      author={Alonso-Guti\'errez, David},
      author={Gonz\'alez~Merino, Bernardo},
      author={Jim\'enez, C.~Hugo},
      author={Villa, Rafael},
       title={Rogers-{S}hephard inequality for log-concave functions},
        date={2016},
        ISSN={0022-1236,1096-0783},
     journal={J. Funct. Anal.},
      volume={271},
      number={11},
       pages={3269\ndash 3299},
         url={https://doi.org/10.1016/j.jfa.2016.09.005},
      review={\MR{3554706}},
}

\bib{AG21}{article}{
      author={Alonso-Guti\'errez, David},
      author={Hern\'andez~Cifre, Mar\'ia~A.},
      author={Roysdon, Michael},
      author={Yepes~Nicol\'as, Jes\'us},
      author={Zvavitch, Artem},
       title={On {R}ogers-{S}hephard type inequalities for general measures},
        date={2021},
        ISSN={1073-7928,1687-0247},
     journal={Int. Math. Res. Not. IMRN},
      number={10},
       pages={7224\ndash 7261},
         url={https://doi.org/10.1093/imrn/rnz010},
      review={\MR{4259147}},
}

\bib{ArtsteinETAL15}{article}{
      author={Artstein-Avidan, S.},
      author={Einhorn, K.},
      author={Florentin, D.~I.},
      author={Ostrover, Y.},
       title={On {G}odbersen's conjecture},
        date={2015},
        ISSN={0046-5755,1572-9168},
     journal={Geom. Dedicata},
      volume={178},
       pages={337\ndash 350},
         url={https://doi.org/10.1007/s10711-015-0060-1},
      review={\MR{3397498}},
}

\bib{ArtsteinPutterman25}{article}{
      author={Artstein-Avidan, S.},
      author={Putterman, E.},
       title={On unbalanced difference bodies and {G}odbersen's conjecture},
        date={2025},
        ISSN={0002-9939,1088-6826},
     journal={Proc. Amer. Math. Soc.},
      volume={153},
      number={12},
       pages={5345\ndash 5359},
         url={https://doi.org/10.1090/proc/17428},
      review={\MR{4989647}},
}

\bib{ArtsteinSadovskySanyal23}{article}{
      author={Artstein-Avidan, S.},
      author={Sadovsky, S.},
      author={Sanyal, R.},
       title={Geometric inequalities for anti-blocking bodies},
        date={2023},
        ISSN={0219-1997,1793-6683},
     journal={Commun. Contemp. Math.},
      volume={25},
      number={3},
       pages={Paper No. 2150113, 30},
         url={https://doi.org/10.1142/S0219199721501133},
      review={\MR{4568892}},
}

\bib{BianchiniColesanti}{article}{
      author={Bianchini, C.},
      author={Colesanti, A.},
       title={A sharp {R}ogers and {S}hephard inequality for the
  {$p$}-difference body of planar convex bodies},
        date={2008},
        ISSN={0002-9939,1088-6826},
     journal={Proc. Amer. Math. Soc.},
      volume={136},
      number={7},
       pages={2575\ndash 2582},
         url={https://doi.org/10.1090/S0002-9939-08-09209-5},
      review={\MR{2390529}},
}

\bib{boroczky96}{article}{
      author={B\"or\"oczky, K., Jr.},
       title={Around the {R}ogers-{S}hepard inequality},
        date={1996},
        ISSN={0865-2090},
     journal={Math. Pannon.},
      volume={7},
      number={1},
       pages={113\ndash 130},
      review={\MR{1378125}},
}

\bib{boroczky05}{article}{
      author={B\"or\"oczky, K., Jr.},
       title={The stability of the {R}ogers-{S}hephard inequality and of some
  related inequalities},
        date={2005},
        ISSN={0001-8708,1090-2082},
     journal={Adv. Math.},
      volume={190},
      number={1},
       pages={47\ndash 76},
         url={https://doi.org/10.1016/j.aim.2003.11.015},
      review={\MR{2104905}},
}

\bib{BuragoZalgaller}{book}{
      author={Burago, Yu.\~D.},
      author={Zalgaller, V.~A.},
       title={Geometric inequalities},
      series={Grundlehren der mathematischen Wissenschaften [Fundamental
  Principles of Mathematical Sciences]},
   publisher={Springer-Verlag, Berlin},
        date={1988},
      volume={285},
        ISBN={3-540-13615-0},
         url={https://doi.org/10.1007/978-3-662-07441-1},
        note={Translated from the Russian by A. B. Sosinski\u i, Springer
  Series in Soviet Mathematics},
      review={\MR{936419}},
}

\bib{Chakerian}{article}{
      author={Chakerian, G.~D.},
       title={Inequalities for the difference body of a convex body},
        date={1967},
        ISSN={0002-9939,1088-6826},
     journal={Proc. Amer. Math. Soc.},
      volume={18},
       pages={879\ndash 884},
         url={https://doi.org/10.2307/2035131},
      review={\MR{218972}},
}

\bib{Colesanti06}{article}{
      author={Colesanti, Andrea},
       title={Functional inequalities related to the {R}ogers-{S}hephard
  inequality},
        date={2006},
        ISSN={0025-5793},
     journal={Mathematika},
      volume={53},
      number={1},
       pages={81\ndash 101},
         url={https://doi.org/10.1112/S0025579300000048},
      review={\MR{2304054}},
}

\bib{develin}{article}{
      author={Develin, Mike},
       title={Disjoint faces of complementary dimension},
        date={2004},
        ISSN={0138-4821},
     journal={Beitr\"age Algebra Geom.},
      volume={45},
      number={2},
       pages={463\ndash 464},
      review={\MR{2093178}},
}

\bib{FaryRedei}{article}{
      author={F\'ary, I.},
      author={R\'edei, L.},
       title={Der zentralsymmetrische {K}ern und die zentralsymmetrische
  {H}\"ulle von konvexen {K}\"orpern},
        date={1950},
        ISSN={0025-5831,1432-1807},
     journal={Math. Ann.},
      volume={122},
       pages={205\ndash 220},
         url={https://doi.org/10.1007/BF01342966},
      review={\MR{39294}},
}

\bib{firey62}{article}{
      author={Firey, Wm.~J.},
       title={{$p$}-means of convex bodies},
        date={1962},
        ISSN={0025-5521,1903-1807},
     journal={Math. Scand.},
      volume={10},
       pages={17\ndash 24},
         url={https://doi.org/10.7146/math.scand.a-10510},
      review={\MR{141003}},
}

\bib{Fradelizi:Equality}{article}{
      author={Fradelizi, M.},
      author={Manui, A.},
      author={Meyer, M.},
      author={Ndiaye, C.~S.},
       title={Equality cases for the {$L_p$-Rogers--Shephard} inequality in the
  plane and for locally anti-blocking bodies in {$\mathbb R^n$}},
        date={2026},
     journal={preprint},
      eprint={arXiv:2606.07887},
}

\bib{Fradelizi:LpRS}{article}{
      author={Fradelizi, M.},
      author={Manui, A.},
      author={Meyer, M.},
      author={Ndiaye, C.~S.},
       title={{$L_p$}-{R}ogers--{S}hephard type inequalities for
  {$L_p$}-zonoids and symmetric bodies},
        date={2026},
     journal={preprint},
      eprint={arXiv:2607.03582},
}

\bib{FK26}{article}{
      author={Fryš, F.},
      author={Kotrbat\'y, J.},
       title={Around higher-order godbersen conjectures},
        date={2026},
     journal={preprint},
      eprint={arXiv:2609.08612},
}

\bib{Godbersen}{book}{
      author={Godbersen, C.},
       title={Der {S}atz vom {V}ektorbereich in {R}äumen beliebiger
  {D}imensionen},
   publisher={Ph.D. Thesis, Göttingen},
        date={1938},
}

\bib{Grunbaum63}{incollection}{
      author={Gr\"unbaum, B.},
       title={Measures of symmetry for convex sets},
        date={1963},
   booktitle={Proc. {S}ympos. {P}ure {M}ath., {V}ol. {VII}},
   publisher={Amer. Math. Soc., Providence, RI},
       pages={233\ndash 270},
      review={\MR{156259}},
}

\bib{Grunbaum2003}{book}{
      author={Gr\"unbaum, Branko},
       title={Convex polytopes},
     edition={Second},
      series={Graduate Texts in Mathematics},
   publisher={Springer-Verlag, New York},
        date={2003},
      volume={221},
        ISBN={0-387-00424-6; 0-387-40409-0},
         url={https://doi.org/10.1007/978-1-4613-0019-9},
        note={Prepared and with a preface by Volker Kaibel, Victor Klee and
  G\"unter M.\ Ziegler},
      review={\MR{1976856}},
}

\bib{HaddadETAL23}{article}{
      author={Haddad, J.},
      author={Langharst, D.},
      author={Putterman, E.},
      author={Roysdon, M.},
      author={Ye, D.},
       title={Affine isoperimetric inequalities for higher-order projection and
  centroid bodies},
        date={2025},
        ISSN={0025-5831,1432-1807},
     journal={Math. Ann.},
      volume={393},
      number={1},
       pages={1073\ndash 1121},
         url={https://doi.org/10.1007/s00208-025-03271-x},
      review={\MR{4966578}},
}

\bib{HaddadETAL25}{article}{
      author={Haddad, J.},
      author={Langharst, D.},
      author={Putterman, E.},
      author={Roysdon, M.},
      author={Ye, D.},
       title={Higher-order {$L^p$} isoperimetric and {S}obolev inequalities},
        date={2025},
        ISSN={0022-1236,1096-0783},
     journal={J. Funct. Anal.},
      volume={288},
      number={2},
       pages={Paper No. 110722, 45},
         url={https://doi.org/10.1016/j.jfa.2024.110722},
      review={\MR{4815903}},
}

\bib{HaddadETAL25+}{article}{
      author={Haddad, Juli\'an},
      author={Langharst, Dylan},
      author={Livshyts, Galyna~V.},
      author={Putterman, Eli},
       title={On the polar of {S}chneider's difference body},
        date={2025},
     journal={preprint},
      eprint={arXiv:2503.06191},
}

\bib{K25}{article}{
      author={Kotrbat\'y, J.},
       title={On a generalization of {G}odbersen's conjecture},
        date={2025},
        ISSN={1073-7928,1687-0247},
     journal={Int. Math. Res. Not. IMRN},
      number={22},
       pages={Paper No. rnaf341, 11},
         url={https://doi.org/10.1093/imrn/rnaf341},
      review={\MR{4988320}},
}

\bib{KW22}{article}{
      author={Kotrbat\'y, Jan},
      author={Wannerer, Thomas},
       title={On mixed {H}odge-{R}iemann relations for translation-invariant
  valuations and {A}leksandrov-{F}enchel inequalities},
        date={2022},
        ISSN={0219-1997,1793-6683},
     journal={Commun. Contemp. Math.},
      volume={24},
      number={7},
       pages={Paper No. 2150049, 24},
         url={https://doi.org/10.1142/S0219199721500498},
      review={\MR{4476308}},
}

\bib{Langharst:Comments}{article}{
      author={Langharst, D.},
       title={Some comments on the mth-order projection bodies},
        date={2026},
     journal={J. Convex Anal.},
       pages={to appear},
}

\bib{LangharstETAL:mOrder}{article}{
      author={Langharst, D.},
      author={Putterman, E.},
      author={Roysdon, M.},
      author={Ye, D.},
       title={On the {$m$}th-order weighted projection body operator and
  related inequalities},
        date={2025},
        ISSN={2189-3756,2189-3764},
     journal={Pure Appl. Funct. Anal.},
      volume={10},
      number={5},
       pages={1323\ndash 1353},
      review={\MR{5000527}},
}

\bib{LangharstRoysdonZhao25}{article}{
      author={Langharst, D.},
      author={Roysdon, M.},
      author={Zhao, Y.},
       title={On the {$m$}th-order affine {P}\'olya-{S}zeg\"o{} principle},
        date={2025},
        ISSN={1050-6926,1559-002X},
     journal={J. Geom. Anal.},
      volume={35},
      number={7},
       pages={Paper No. 205, 32},
         url={https://doi.org/10.1007/s12220-025-02039-8},
      review={\MR{4913078}},
}

\bib{LangharstSolaUlivelli24}{article}{
      author={Langharst, D.},
      author={Sola, F.~M.},
      author={Ulivelli, J.},
       title={Higher-order reverse isoperimetric inequalities for log-concave
  functions},
        date={2024},
     journal={preprint},
      eprint={arXiv:2403.05712},
}

\bib{LangharstXi24}{article}{
      author={Langharst, D.},
      author={Xi, D.},
       title={General higher order {$L^p$} mean zonoids},
        date={2024},
        ISSN={0002-9939,1088-6826},
     journal={Proc. Amer. Math. Soc.},
      volume={152},
      number={12},
       pages={5299\ndash 5311},
         url={https://doi.org/10.1090/proc/16914},
      review={\MR{4856394}},
}

\bib{LutwakI}{article}{
      author={Lutwak, E.},
       title={The {B}runn-{M}inkowski-{F}irey theory. {I}. {M}ixed volumes and
  the {M}inkowski problem},
        date={1993},
        ISSN={0022-040X,1945-743X},
     journal={J. Differential Geom.},
      volume={38},
      number={1},
       pages={131\ndash 150},
         url={http://projecteuclid.org/euclid.jdg/1214454097},
      review={\MR{1231704}},
}

\bib{LutwakII}{article}{
      author={Lutwak, E.},
       title={The {B}runn-{M}inkowski-{F}irey theory. {II}. {A}ffine and
  geominimal surface areas},
        date={1996},
        ISSN={0001-8708,1090-2082},
     journal={Adv. Math.},
      volume={118},
      number={2},
       pages={244\ndash 294},
         url={https://doi.org/10.1006/aima.1996.0022},
      review={\MR{1378681}},
}

\bib{LYZ12}{article}{
      author={Lutwak, E.},
      author={Yang, D.},
      author={Zhang, G.},
       title={The {B}runn-{M}inkowski-{F}irey inequality for nonconvex sets},
        date={2012},
        ISSN={0196-8858,1090-2074},
     journal={Adv. in Appl. Math.},
      volume={48},
      number={2},
       pages={407\ndash 413},
         url={https://doi.org/10.1016/j.aam.2011.11.003},
      review={\MR{2873885}},
}

\bib{Makai}{article}{
      author={{Makai Jr.}, E.},
       title={Research problems},
        date={1974},
        ISSN={0031-5303,1588-2829},
     journal={Period. Math. Hungar.},
      volume={5},
      number={4},
       pages={353\ndash 354},
         url={https://doi.org/10.1007/BF02018191},
      review={\MR{1553587}},
}

\bib{ManuiNdiayeZvavitch24}{article}{
      author={Manui, A.},
      author={Ndiaye, C.~S.},
      author={Zvavitch, A.},
       title={On the volume of sums of anti-blocking bodies},
        date={2026},
     journal={Commun. Contemp. Math.},
       pages={to appear},
}

\bib{RogersShephard57}{article}{
      author={Rogers, C.~A.},
      author={Shephard, G.~C.},
       title={The difference body of a convex body},
        date={1957},
        ISSN={0003-889X,1420-8938},
     journal={Arch. Math. (Basel)},
      volume={8},
       pages={220\ndash 233},
         url={https://doi.org/10.1007/BF01899997},
      review={\MR{92172}},
}

\bib{RogersShephard58}{article}{
      author={Rogers, C.~A.},
      author={Shephard, G.~C.},
       title={Convex bodies associated with a given convex body},
        date={1958},
        ISSN={0024-6107,1469-7750},
     journal={J. London Math. Soc.},
      volume={33},
       pages={270\ndash 281},
         url={https://doi.org/10.1112/jlms/s1-33.3.270},
      review={\MR{101508}},
}

\bib{Roysdon20}{article}{
      author={Roysdon, M.},
       title={Rogers-{S}hephard type inequalities for sections},
        date={2020},
        ISSN={0022-247X,1096-0813},
     journal={J. Math. Anal. Appl.},
      volume={487},
      number={1},
       pages={123958, 22},
         url={https://doi.org/10.1016/j.jmaa.2020.123958},
      review={\MR{4070397}},
}

\bib{Sadovsky25}{article}{
      author={Sadovsky, S.},
       title={Godbersen's conjecture for locally anti-blocking bodies},
        date={2025},
        ISSN={1615-715X,1615-7168},
     journal={Adv. Geom.},
      volume={25},
      number={3},
       pages={307\ndash 315},
         url={https://doi.org/10.1515/advgeom-2025-0019},
      review={\MR{4933586}},
}

\bib{Schneider70}{article}{
      author={Schneider, R.},
       title={Eine {V}erallgemeinerung des {D}ifferenzenk\"orpers},
        date={1970},
        ISSN={0026-9255,1436-5081},
     journal={Monatsh. Math.},
      volume={74},
       pages={258\ndash 272},
         url={https://doi.org/10.1007/BF01303445},
      review={\MR{262926}},
}

\bib{Schneider85}{incollection}{
      author={Schneider, R.},
       title={On the {A}leksandrov-{F}enchel inequality},
        date={1985},
   booktitle={Discrete geometry and convexity ({N}ew {Y}ork, 1982)},
      series={Ann. New York Acad. Sci.},
      volume={440},
   publisher={New York Acad. Sci., New York},
       pages={132\ndash 141},
         url={https://doi.org/10.1111/j.1749-6632.1985.tb14547.x},
      review={\MR{809200}},
}

\bib{Schneider00}{incollection}{
      author={Schneider, R.},
       title={Mixed functionals of convex bodies},
        date={2000},
      volume={24},
       pages={527\ndash 538},
         url={https://doi.org/10.1007/s004540010054},
        note={The Branko Gr\"unbaum birthday issue},
      review={\MR{1758069}},
}

\bib{Schneider09}{article}{
      author={Schneider, R.},
       title={Stability for some extremal properties of the simplex},
        date={2009},
        ISSN={0047-2468,1420-8997},
     journal={J. Geom.},
      volume={96},
      number={1-2},
       pages={135\ndash 148},
         url={https://doi.org/10.1007/s00022-010-0028-0},
      review={\MR{2659606}},
}

\bib{Schneider2014}{book}{
      author={Schneider, R.},
       title={Convex bodies: the {B}runn-{M}inkowski theory},
     edition={expanded},
      series={Encyclopedia of Mathematics and its Applications},
   publisher={Cambridge University Press, Cambridge},
        date={2014},
      volume={151},
        ISBN={978-1-107-60101-7},
      review={\MR{3155183}},
}

\bib{vanHandelWang}{article}{
      author={van Handel, R.},
      author={Wang, S.},
       title={On {M}inkowski's monotonicity problem},
        date={2025},
     journal={preprint},
      eprint={arXiv:2507.20082},
}

\bib{ZhouETAL25}{article}{
      author={Zhou, X.},
      author={Ye, D.},
      author={Zhang, Z.},
       title={The {$m$}th order {O}rlicz projection bodies},
        date={2025},
        ISSN={1050-6926,1559-002X},
     journal={J. Geom. Anal.},
      volume={35},
      number={9},
       pages={Paper No. 259, 35},
         url={https://doi.org/10.1007/s12220-025-02082-5},
      review={\MR{4929667}},
}

\bib{Ziegler1995}{book}{
      author={Ziegler, G\"unter~M.},
       title={Lectures on polytopes},
      series={Graduate Texts in Mathematics},
   publisher={Springer-Verlag, New York},
        date={1995},
      volume={152},
        ISBN={0-387-94365-X},
         url={https://doi.org/10.1007/978-1-4613-8431-1},
      review={\MR{1311028}},
}

\end{biblist}
\end{bibdiv}

\end{document}